\documentclass[xcolor,usenames,dvipsnames,10pt]{article}
\usepackage{graphicx,pstricks,latexsym}
\usepackage[utf8]{inputenc}

\usepackage{makeidx}
\usepackage{amsmath,amsfonts,amssymb,mathtools,multicol}
\usepackage{mathrsfs,bbm,times}
\usepackage{esint}
\usepackage{enumitem}
\catcode`\@=11
\@addtoreset{equation}{section}
\catcode`\@=12
\usepackage{amsthm}
\usepackage{enumitem}

\catcode`\@=11
\@addtoreset{equation}{section}
\catcode`\@=12
\usepackage{amsthm}
\usepackage{enumitem}

\numberwithin{equation}{section}
\theoremstyle{plain}
\newtheorem{theorem}{Theorem}[section]

\newtheorem{corollary}[theorem]{Corollary}

\newtheorem{definition}[theorem]{Definition}

\newtheorem{lemma}[theorem]{Lemma}

\newtheorem{prop}[theorem]{Proposition}

\theoremstyle{definition}
\newtheorem{remark}[theorem]{Remark}

\newtheorem*{remark*}{Remark}
\newtheorem*{defn*}{Definition}
\newtheorem{example}[theorem]{Example}
\newtheorem*{example*}{Example}
\newtheorem*{theorem*}{Theorem}
\newtheorem*{corollary*}{Corollary}
\newtheorem*{result*}{Result}
\newtheorem*{lemma*}{Lemma}
\newtheorem*{def*}{Definition}
\newtheorem*{definitionh*}{Definition $\boldsymbol{(\sH)}$}
\newtheorem*{whitney*}{H. Whitney}
\newtheorem*{knaster*}{B. Knaster}
\newtheorem*{prop*}{Proposition}

\numberwithin{equation}{section}

\newtheorem*{A*}{Example A}
\newtheorem*{B*}{Example B}
\newtheorem*{C*}{Example C}

\newcommand{\Om}{\Omega}
\newcommand{\om}{\omega}
\newcommand{\ol}{\overline}
\renewcommand{\forall}{\text{ for all }}
\def\NN{\mathbb N}

\newcommand{\x}{\times}
\newcommand{\wk}{\rightharpoonup}
\newcommand{\pd}{\partial}
\renewcommand{\leq}{\leqslant}
\renewcommand{\geq}{\geqslant}
\def\RR{\mathbb R}

\newcommand{\e}{\epsilon}

\newcommand{\vt}{\vartheta}

\renewcommand{\a}{\alpha}

\newcommand{\sT}{\mathcal{T}}
\newcommand{\half}{{\frac{1}{2}}}

\renewcommand{\L}{\Lambda}
\renewcommand{\l}{\lambda}
\newcommand{\sH}{\mathcal{H}}
\newcommand{\sS}{\mathcal{S}}

\newcommand{\vp}{\varphi}

\newcommand{\sG}{\mathcal{G}}

\renewcommand{\b}{\beta}

\newcommand{\ZZ}{\mathbb{Z}}
\newcommand{\sA}{\mathcal{A}}

\newcommand{\sK}{\mathcal K}
\newcommand{\sC}{\mathcal{C}}
\newcommand{\ra}{\rightarrow}

\newcommand{\uhalf}{{\scriptstyle{\frac12}}}

\newcommand{\uf}{{\scriptstyle{\frac14}}}

\newcommand{\sL}{\mathcal L}

\renewcommand{\iff}{\text{ if and only if }}

\newcommand{\ed}{\end{document}}

\title{Path-Connectedness in Global Bifurcation Theory}

\author{J. F. Toland}
\date{}
\begin{document}

\maketitle
\begin{abstract}\noindent A celebrated result in bifurcation theory  is that global connected sets of non-trivial solutions
bifurcate from trivial solutions at non-zero eigenvalues of odd algebraic multiplicity of the linearized problem when the  operators involved are compact. In this paper a simple example is constructed  which satisfies the regularity hypotheses of the global bifurcation theorem, and the eigenvalue  has  algebraic multiplicity one, yet all the path-connected components of the  connected sets
that bifurcate are singletons.  Another example shows that even when the operators are everywhere infinitely differentiable and classical  bifurcation occurs locally at a simple eigenvalue, the global continuum may not be path-connected away from the bifurcation point. A third example shows that the  non-trivial solutions which, by variational theory,  bifurcate from  eigenvalues of any multiplicity when the problem has gradient structure, may not be connected and may contain no paths except singletons.
\end{abstract}
{MSC2020: 47J15, 58E07, 35B32, 54F15}\\
{Key Words: Global Bifurcation; Path-connectedness;\\ Hereditarily Indecomposable Continua;\\
Topological Degree and  Variational methods.}
\section{Introduction}
 Krasnosel'skii \cite{kras} considered non-linear eigenvalues in the form
\begin{subequations}\label{overall}\begin{equation} \label{eqn}\l x = Lx + R(\l,x), \quad \l \in \RR,~~ x \in X,
\end{equation}
where $X$ is a real Banach space,
the linear operator
$L:X \to X$ is compact,  and the non-linear  $R:\RR \times X \to X$ is continuous, compact, and  satisfies
\begin{equation}\label{hot}\frac{\|R(\l, x)\|}{\|x\|} \to 0 \text{ as } 0 \neq \|x\| \to 0 \text{ uniformly for $\l$ in bounded sets.}
\end{equation}
\end{subequations}
Since $R$ is continuous, \eqref{hot} implies that $R(\l,0)=0$, and hence $x=0$ is a solution of \eqref{eqn}, for all $\l \in \RR$. Let $\sT = \{(\l,0):\l \in \RR\}$ denote the set of trivial solutions of \eqref{eqn} and   $\sS$ the set of solutions that are not trivial. In all that follows, $L$ is linear and compact and $R$ is nonlinear, continuous and compact.
The first observation  is that under these hypotheses $\sS$ may be empty.
 \begin{example}\label{kexx}
Let $X = \RR^2$,
$L(x,y) = (x+y,y)$ and  $R(x,y) = (0,-x^3)$. Then $L$ is linear, \eqref{hot} holds, and
equation \eqref{eqn} is satisfied if and only if
$$(\l-1)x=y \text{ and }(\l-1)y = -x^3,$$
which implies that
$x\big( (\l-1)^2 + x^2\big) = 0$. Hence $x=0$ and, by the first equation, $y=0$, which shows  $\sS = \emptyset$.\qed
\end{example}

According to Krasnosel'skii \cite[Ch. IV, p. 181]{kras}, a point $\l_0 \in \RR$ is a bifurcation point for \eqref{eqn} if there exists a sequence $\{(\l_n, x_n)\}\subset \sS$ such that
$$ \l_n \to \l_0 \text{~ in ~ }\RR\quad \text{ and }\quad 0\neq \|x_n\| \to 0 \text { as } n \to \infty.$$
In this definition there is no mention of curves, or even  connected sets in $\sS$, bifurcating from $\sT$ at $(\l_0,0)$.

\emph{ By a path of  solutions in $\sS$ is meant $\{\gamma(t): t \in [0,1]\}$  where $\gamma:[0,1]\to \sS \subset \RR\times X$ is continuous. A path is non-trivial if $\gamma$ is not constant, and a curve is a smooth path.\qed}

The following  necessary criterion for $\l_0$ to be a bifurcation points
when $L$ is compact and $R$ satisfies \eqref{hot} in a Banach space $X$  is well known \cite[Ch.\,IV,\,\S 2, Lem. 2.1]{kras}.

\emph{A real number $\l_0\neq 0$ is a bifurcation point
only if it is an eigenvalue of $L$.   If $X$ is finite-dimensional and $\l_0 = 0$ is a bifurcation point, then 0 is an eigenvalue of $L$.}\qed

Note from Example \ref{kexx} that a real eigenvalue of $L$ need  not  be a be bifurcation point.

\subsection{Bifurcation Theory - Background}

\paragraph{Multiplicities.} \emph{The   geometric multiplicity of an eigenvalue  $\l_0$ of $L$ is the  dimension of the eigenspace  $\ker(\l_0I-L)$, and  its algebraic multiplicity is the dimension of $\cup_{k\in \NN}\ker (\l_0I-L)^k$.
When the algebraic multiplicity is one, $\l_0$ is called simple.
The  multiplicities of  all non-zero eigenvalues of  compact operators are finite.\qed}

\subsubsection*{From Classical Analysis}
Many seemingly different bifurcation phenomena were studied in \emph{ad hoc} situations before being recognised  by  Crandall \& Rabinowitz \cite{cr} as  special cases of the following overarching result, which is a corollary of  the Implicit Function Theorem.
\begin{theorem}\label{simple}
{\bf Bifurcation from a simple eigenvalue} \textnormal{\cite{cr}} Suppose  that  $\l_0$ is a simple eigenvalue of $L$,  that
$R:\RR \times X \to X$ is continuously differentiable, and that $\pd_{x\l}R$ exists and is continuous in a neighbourhood of $(\l_0,0)$.
Then there is an injective, continuous function $\gamma:(-1,1) \to \RR \times X$ such that $\gamma (0) = (\l_0,0)$ and a neighbourhood $U$ of $(\l_0,0)$ such that $U \cap \sS = \{\gamma(s):s\in (-1,0)\cup (0,1)\}$.
If $\pd_{xx}R$ is also continuous in the neighbourhood of $(\l_0,0)$, then $\gamma$ is $C^1$. \end{theorem}
\begin{remark*} The nonlinearity $R$ in Example \ref{kexx} satisfies the hypotheses in Theorem \ref{simple}, and  the only eigenvalue of $L$ is $\l_0=1$ which has geometric multiplicity 1, but it is not a  bifurcation point. However
 Theorem \ref{simple} does not apply because  the algebraic   multiplicity of $\l_0$ is 2. \qed\end{remark*}
\emph{From now on, the word multiplicity will always refer to algebraic multiplicity.}

\subsubsection*{From Topological Degree Theory}
In 1950, Krasnosel'skii \cite[Thm.\,2]{krasn}, \cite[\S IV.2,\,p.\,196]{kras}, used  Leray-Schauder topological degree theory to prove, under the hypotheses of \eqref{overall},  that every non-zero eigenvalue of $L$ with odd multiplicity is a bifurcation point \cite[Ch.\,IV, Thm.\,2.1] {kras}. Then, under the same hypotheses, in 1971 Rabinowitz improved the method to obtain the ground-breaking result that  a connected set of non-trivial solutions bifurcates globally in $\RR \times X$ from $\sT$ at   eigenvalues of odd   multiplicity.
\begin{theorem} \label{godd}{\bf  Global bifurcation at odd   multiplicity eigenvalues}\textnormal{ \cite[Thm.\,1.3]{phr}}. Suppose $L$ and $R$ are as in \eqref{overall}  and $\l_0$ is a non-zero  eigenvalue of $L$ of odd  multiplicity. Then $\l_0$ is a bifurcation point
and there exists a connected subset $\mathcal C$ of $\sS$ such that $(\l_0,0)\in \ol{\sC}$ and
either $\sC$ is unbounded in $\RR \times X$ or there exists  $(\l_1,0) \in \ol{\sC}$ where $\l_1 \neq \l_0$ is also an eigenvalue of odd   multiplicity of $L$.
 (If $X$ is finite-dimensional, the result  holds when $\l_0=0$ is an eigenvalue of   odd    multiplicity.)\end{theorem}

\textbf{Related Results.}  Krasnosel'skii \cite[Ch.\,IV.5,\,p.232\,ff.]{kras} also studied bifurcation at eigenvalues of even  multiplicity, when the nonlinearity $R$ is non-degenerate in a certain sense. Under his hypotheses, the methods introduced by Rabinowitz \cite{phr} lead easily to global bifurcation at eigenvalues of even multiplicity for certain nonlinearities  \cite{jft2}. The   Examples in Section \ref{caut}  are relevant in that context also. \qed

\subsubsection*{ From Variational Methods}
To justify linearisation without reference to multiplicity of eigenvalues, Krasnoselskii \cite[\S VI]{kras} developed a variational approach to bifurcation theory in Hilbert space.
(For an up-to-date  account in Banach spaces, see \cite{phrab}.)
Let $X$ be a real Hilbert space with inner product $\langle\cdot, \cdot\rangle$ and let $h:X \to \RR$ be differentiable with derivative  $Dh[x]:X \to \RR$ at $x \in X$. Then $Dh[x]$ is a  bounded linear operator on $X$ and, by the Riesz Representation Theorem, there exists a unique $x^* \in X$ such $Dh[x]y = \langle x^*,  y\rangle$ for all $y \in X$. Hence
$\nabla h(x) = x^*$  defines an operator $\nabla h:X \to X$,  called the gradient of $h$, and an operator $H:X \to X$ is said to have gradient structure  if $H = \nabla h$ for some differentiable $h:X \to \RR$.

It is easily  seen  that a bounded linear operator $L:X \to X$ has gradient structure if and only if $\langle Lx,y\rangle = \langle x,Ly\rangle$ for all $x,y \in X$. In other words $L$ is a gradient if and only if it is self adjoint, in which case $Lx = \nabla \ell (x)$ where $\ell(x) = \uhalf \langle Lx,x\rangle,\,x\in X$.  Note that  when $L$ is self-adjoint,   $(L-\l I)^2 x = 0$ implies
$$
\|(L-\l I) x\|^2 =\langle (L-\l I)x,(L-\l I)x\rangle = \langle (L-\l I)^2x,x\rangle = 0,
$$
and hence algebraic multiplicity and geometric multiplicity of eigenvalues coincide for self-adjoint operators. Obviously the identity operator $I$ on $X$ has gradient structure  $ I= \nabla \iota$ where $\iota(x) = \uhalf \|x\|^2$.
Finally, a function $g:X \to \RR$  is weakly continuous  if  $g(x_k) \to g(x_0)$ in $\RR$ for every weakly convergent sequence $x_k\wk x_0$ in $X$. Vainberg proved \cite[Thm. 8.2]{vainberg} that $g$ is weakly continuous when its
gradient is a compact operator.
In this setting, a special case of \eqref{eqn}  is
\begin{subequations}\label{gradprob}\begin{equation} \label{gradeqn}\l x = Lx + R(x), \quad \l \in \RR,~ x \in X,
\end{equation}
where $X$ is a   real Hilbert space, $L:X \to X$ is self-adjoint, and $R$ satisfies \eqref{hot} with  gradient structure     independent of $\l$:
\begin{equation}\label{ghot}
R(x) = \nabla r (x), \text{ where } r \text{ is weakly continuous}.
\end{equation}
\end{subequations}
Krasnosel'skii proved   \cite[Ch. VI, \S 6, Thm.\,2.2, p.\,332]{kras} that
for this class of problems bifurcation occurs at all non-zero eigenvalues of $L$, independent of multiplicity.
The following version of his theorem replaces his hypothesis that ``$R$ is uniform differentiable'' near 0  with Vainberg's  characterisation \cite[Thm.\,4.2,\,p.\,45]{vainberg} of uniform differentiability  in terms of
the bounded uniform continuity of its Fr\'echet derivative.

\begin{theorem}{\bf With gradient structure, non-zero eigenvalues are bifurcation points}
\label{gradbifthy}If  $R$  in \eqref{gradprob}   has Fr\'echet derivative $x \mapsto dR[x]$  bounded and uniformly continuous in a neighbourhood of 0 in $X$, all eigenvalues $\l_0\neq 0$ of $L$ are bifurcation points.
(When $X$ is finite-dimensional, the condition $\l_0 \neq 0$ is not needed.) \end{theorem}

\begin{remark}While Rabinowitz's  theory of global bifurcation yields  \emph{connected} sets $\sC \subset \sS$ bifurcating from $\sT$ at eigenvalues of odd multiplicity,   B\"ohme's example \cite[\S 6]{bohme}
 showed that
no such connectedness is guaranteed by Theorem \ref{gradbifthy}.\qed
\end{remark}
\subsubsection*{From   Real-Analyticity}
So far in this summary,  Theorem \ref{simple} is the only result which guarantees the existence of a curve of non-trivial solutions of equation \eqref{overall}, and even then it is localized to a neighbourhood $(\l_0,0)$, where $\l_0$ is the bifurcation point.
However, in 1973 Dancer \cite{end1,end2} showed, among many other things, that when the operators in \eqref{eqn} are real-analytic (infinitely differentiable and equal to the sum of their Taylor series about every point), there \emph{bifurcates from a simple eigenvalue a global path  of solutions which, at every point, has a local real-analytic re-parametrization}.
More precisely, provided some standard
functional-analytic structure is in place,
 the global continuum $\sC$  which
bifurcates from the trivial solutions at a simple eigenvalue
contains a continuous curve $\sK$, parameterised by $s \in [0,\infty)$, with  the following properties.
\begin{itemize}
\item[(i)]  $\sK =\{(\L(s), \kappa (s)): s \in [0,
\infty)\}\subset \sC$  is either unbounded or forms a closed loop
in $\RR \x X$.
\item[(ii)]
For each $s^* \in (0,\infty)$ there exists $\rho^*:(-1,1) \ra \RR$
(a re-parametrisation) which is continuous, injective, and
\begin{align*} \rho ^*(0) = s^* ~\text{ and }~~
t \mapsto (\L(\rho^*(t)), \kappa (\rho^*(t))
\textnormal{ is analytic on} ~~ t \in (-1,1).
\end{align*} This does not imply that $\sK$ is locally a smooth
curve. (The map $\sigma:(-1,1) \ra \RR^2$ given by $\sigma (t) =
(t^2,t^3)$ is real-analytic and its image is a curve with a cusp.)
Nor does it preclude the possibility of secondary bifurcation
points on $\sK$. In particular, since $(\Lambda,\kappa):[0,\infty)
\ra \RR \x X$ is not required to be globally injective;
self-intersection of $\sK$ (as in a figure eight) is not ruled
out.
\item[(iii)]
Secondary bifurcation points and points where the  bifurcating branch is not smooth, if any,
are isolated. \end{itemize}

In plain language, by Dancer's theory  for real-analytic equations (see \cite{bt} for an introductory account) there
bifurcates from a simple eigenvalue a unique global path of solutions which  is a smooth 1-dimensional manifold except possibly at a discrete set of  points. This path is unique in the sense that it has a pre-determined  continuation
through secondary bifurcation points, or even when it encounters a higher-dimensional manifold of solutions. \qed

The following three example illustrate how  the situation may be radically different when the hypotheses of Theorems \ref{simple}, \ref{godd} and \ref{gradbifthy}   are satisfied
but the operators inbolved are infinitely differentiable, but not real-analytic.

The main conclusion, that
  all path-connected components of the set of non-trivial solutions may be singletons, and the resulting possibility that in principal  no two non-trivial solutions can be joined by a path, are obviously   important for applications.

\section{Three Examples}\label{caut}

In the first two examples of \eqref{overall}, $X = \RR$,  $L=0$, $\l_0 = 0$ is the only  eigenvalue of $L$ and is simple,
and $R= r:\RR^2 \to \RR$, where $r$  is at least continuously differentiable  and satisfies \eqref{hot}. Under these hypotheses
\eqref{eqn} has the form
\begin{equation}\label{EQN}\l x = r(\l,x),\qquad (\l,x) \in \RR^2.\end{equation}
The first example concerns  global connected sets of solutions of \eqref{EQN} which  by Theorem \ref{godd}  bifurcate at the
 simple eigenvalue $\l_0=0$ when Theorem \ref{simple} does not apply because $\pd_{\l x}r$ is not continuous at $(0,0)$.
\begin{A*} 
There exists a $C^1$-function $r$  on $\RR^2$ which is  infinitely differentiable on $\RR^2\setminus \{(0,0)\}$,  for which  the closure $\ol \sC$ of the global connected sets  of non-trivial solutions of \eqref{EQN}, which bifurcates at $\l_0= 0$ by  Theorem \ref{godd},  is unbounded and connected but contains  no path-connected sets except singletons.\qed
\end{A*}
The second example illustrates the possible behaviour of  solutions which bifurcate at  $\l_0=0$ when simultaneously  Theorems \ref{simple} yields   the local bifurcation  of a smooth curve, and Theorem  \ref{godd} yields the global bifurcation of unbounded connected sets of non-trivial solutions of \eqref{EQN}.
\begin{B*}
There exists an infinitely differentiable function $r:\RR^2 \to \RR$  for which
the closure $\ol \sC$ of the  connected sets  of non-trivial solutions of \eqref{EQN} that bifurcate at $\l_0=0$,   by Theorem \ref{simple} and Theorem \ref{godd},   is the union of three disjoint connected sets,
$$ \ol\sC=\sL \cup\sC^+ \cup \sC^-.
$$
Here  $\sL$  is the smooth curve $\{0\}\times (-\uhalf,\uhalf)$, and $\sC^\pm$ are  closed, unbounded, connected sets  in the first and third quadrants respectively with $(\pm \uhalf,0) \in \sC^\pm$ and all  path-connected components of $\sC^+\cup \sC^-$ are singletons. (The only non-trivial paths in $\ol \sC$ are subsets of the closure of $\sL$.)\qed
\end{B*}
Although B\"ohme  \cite{bohme} showed that the non-trivial solution set of \eqref{gradprob} given by Theorem \ref{gradbifthy} when $R$ has gradient structure need not be connected, it may have   connected components. The example below is constructed to illustrates  the possibility that in addition all its  path-connected components may be singletons.

\begin{C*} In this  example of problem \eqref{gradprob}, $X = \RR^2$,  $R = \nabla r$ where   $r:\RR^2 \to \RR$ is infinitely differentiable, and  $L$ is the zero operator which has only one eigenvalue, namely 0 with  multiplicity 2. Then  \eqref{gradeqn} has the form
\begin{equation}\label{EQN2}\l (x,y) = \nabla r(x,y),\quad (x,y) \in \RR^2,~\l \in \RR,\end{equation}
and the existence of non-trivial solutions with $\big(\l,( x,y)\big)$ near  $\big(0,(0,0)\big))$ is guaranteed by Theorem \ref{gradbifthy}.
 Example C shows that, in addition to not forming a connected set,  all the path-connected components of the non-trivial solution set may be singletons.\qed
\end{C*}

\subsection*{Key Ingredients}\label{keyi}
The construction of these examples relies on  classical results of   Whitney and Knaster.

\begin{theorem*}  {\textbf{(Whitney)}} \emph{ For any closed set $G \subset\RR^n$ there is an infinitely differentiable, globally Lipschitz continuous function $h$ such that $G= \{x \in \RR^n: h(x) = 0\}$, and  all the derivatives of $h$ are zero at every point of $G$.}
\end{theorem*}
 \begin{proof}
   Let  $u:[0,\infty) \to [0,1]$ be a $C^\infty$-function with
$$ u(t) = 1,~~ t\in [0,\,1/2];\quad u(t)\in (0,1),~~t \in (1/2,1);\quad
u(t) = 0,~~ t \in [1,\infty).$$
  For a closed set $G$, let the open set   $\RR^n\setminus G$ be the union of a countable collection of open balls $\{B_{r_j}(a_j):j \in \NN\}$,  with radius $r_j\in (0,1)$  centred at $a_j \in\RR^n$, and put
$$
u_j(x) = u\left(\frac{|x-a_j|}{r_j}\right),\quad x \in \RR^n.
$$
Then $u_j$, which is infinitely differentiable  on $\RR^n$, is positive on $B_{r_j}(a_j)$ and  supported on $\ol{B_{r_j}(a_j)}$. Let $\gamma_j= \max\{\|D^k u_j\|:0\leq k\leq j\}$ where $D^k u$ is the $k^{th}$ derivative of $u$. Then
 $$D^kh(x)=\sum_{j\in \NN} \frac{ D^k u_j(x)}{\gamma_j\,2^j}
\text{ when }h(x)=\sum_{j\in \NN} \frac{ u_j(x)}{\gamma_j\,2^j}\quad x \in \RR^n,$$
since both series are uniformly convergent.
This shows that $G = \{x\in \RR^n: h(x  ) = 0\}$ and
 $h:\RR^n \to [0,\infty)$ is  $C^\infty$. Now  let $\rho:\RR\to [0,1]$
be such that
$$\rho(x)>0,~~x\neq 0,\quad \rho(0)= 0 \text{ and } \frac{d\rho^k}{dx^k}(0) = 0,~~k \in \NN, $$
and replace $h$ with $\rho\circ h$ to obtain a function $h$ as in the statement of the theorem.\end{proof}

The second ingredient is a deep result in point-set topology due to Knaster \cite{knaster} in 1922.
\begin{def*}
\emph{A {continuum}, which is a  compact, connected set in a metric space, is  indecomposable if it is not a union of two proper sub-continua, and \emph{hereditarily indecomposable} if every sub-continuum
is indecomposable.  (See \cite{bingsnake, burgessa,fearnley, ingram}.)\qed}\end{def*}

\begin{theorem*} {\textbf{(Knaster)}}  \emph{In $\RR^2$ there  exists a hereditarily indecomposable $Q$.} \qed \end{theorem*}
\begin{remark} Since a non-trivial path in $Q$ would be a decomposable sub-continuum, there are no paths in $Q$. In other words,  although $Q$ is compact and connected in $\RR^2$, all its path-connected components  are singletons.\qed
\end{remark}
A  hereditarily indecomposable continuum which is chained (Appendix, Definition \ref{linearchain})  is called a \emph{pseudo-arc} and all pseudo-arcs are homeomorphic \cite[Thm.\,1]{bingconcerning}. Since, by construction, Knaster's $Q$ is chained, it is in a sense the unique pseudo-arc.
For an indication of the significance  of pseudo-arcs in the theory of continua,  see  Remark \ref{linearchain}, but
what  is important here is that $Q$
is compact, connected and contains no paths.
\subsubsection*{Preliminaries}
  Let $Q$ be a pseudo-arc and without loss of generality suppose
$$
Q \subset [0,\pi] \times   [-\uf,\uf],~~Q \cap\big( \{ 0\}\times [-\uf,\uf]\big) \neq \emptyset \text{ and }  Q \cap\big( \{\pi\}\times [-\uf,\uf]\big) \neq \emptyset.
$$
Now let
$P = \{(\l, x\sin \l) \in \RR^2: (\l,x) \in Q\}$. Then $P \subset [0,\pi] \times   [-\uf,\uf]$,
$$
 ~~P \cap\big( \{ 0\}\times [-\uf,\uf]\big) =\{(0, 0)\}, \quad P \cap\big( \{ \pi\}\times [-\uf,\uf]\big) =\{(\pi, 0)\},
$$
and  $P$ is a connected set (being the continuous image of a connected set) which  contains no non-constant
 paths (since $Q$ is  hereditarily indecomposable).

Since $P$ is connected, by  Proposition \ref{prop1} and Corollary \ref{piecewise}, for any $\e>0$ there exists   an ordered set, $\{p^\e_i:1 \leq i \leq n_\e\} \subset P$ such that
$$
p^\e_1 = (0,0),~~ p^\e_n = (\pi,0), ~~ \text{and} \quad \|p^\e_i-p^\e_{i+1}\|<\e \forall 1\leq i\leq n_\e-1,
$$
and the union $L^\e$, of the straight line segments which join the points in order, is a continuous, piecewise-linear, non-self-intersecting curve joining $(0,0)$ to $(\pi,0)$.
Now define subsets of $\RR^2$ by
\begin{equation}\label{Ftilde}
\begin{split} P_k &= P +(k\pi,0), \qquad L_k^\e =L^\e + (k\pi, 0),\\  \widetilde P &= \bigcup_{k\in\ZZ} P_k,\qquad
\qquad \widetilde L^\e = \bigcup_{k\in\ZZ} L^\e_k.\end{split}
 \end{equation}
 and note that $\widetilde  L^\e$ is an unbounded, piecewise linear continuum which separates the plane, and each point of
$\widetilde  L^\e$ is within distance $\e$ of a point of $\widetilde  P$.

 Now let  $\widetilde P_c^\pm$ denote the connected components of
 $\RR^2\setminus \widetilde P$ which contain the half spaces $\{(\l,x): \l\in \RR,~\pm x >\uf\}$, respectively.

\begin{lemma}\label{newlem}
In the plane,   $\widetilde P \subset \RR \times \left(-\uf, \uf\right) $  is an unbounded,  connected subset of a double cone centred on the horizontal axis  with opening angle  $\theta <\pi/6$. Moreover $\widetilde P$ contains no non-trivial paths,~ $(0,0)\in\widetilde P$,~ and ~
 $\widetilde P_c^+ \cap\widetilde P_c^- =\emptyset$.
\end{lemma}
\begin{proof} From the
 definition, $(0,0) \in\widetilde P$ and $\widetilde P \subset \RR \times \left(-\uf, \uf\right) $ is unbounded. Since $P = \{(\l,x\sin \l) \in \RR^2: (\l,x) \in Q\}$ and $|x|<\uf$, $\widetilde P$ lies in a cone with opening angle less than $2\arctan (\uf) <\pi/6$.
Moreover $\widetilde P$ is connected  because $P_k$ is connected and $P_k\cap P_{k+1} = \{ (k+1)\pi,0)\}$ for all $k$, and since
each $P_k$ contains no paths,  a non-trivial path  in $\widetilde P$ must contain
points $(\l_i,x_i)$ with $\l_i$ in the open intervals $(k_i\pi,(k_i+1)\pi)$, $i= 1,2$, where $k_1 \neq k_2$.
However, this  implies that these $P_{k_i}$  contain non-trivial paths, which is false. Hence $\widetilde P$ contains no non-trivial paths.

Now suppose $\widetilde P_c^+ \cap\widetilde P_c^- \neq \emptyset$. Then, since  $\widetilde P_c^+ \cup\widetilde P_c^-$ is open and connected, it is path-connected. Therefore there exists  a path $\gamma$ joining $(0,-2)$  to $(0,2)$ with
$\gamma[0,1]\subset [-K,K] \times [-K, K]$ for some  $K >0$, since $\gamma$ is continuous.
Since, for all $\e>0$, $\widetilde L^\e$ in \eqref{Ftilde} separates the plane, there exists
$$
q_\e \in \gamma\cap  \widetilde L^\e \subset [-K,K] \times [-K, K], \text{ and } p_\e\in \widetilde P \text{ with } \|p_\e -q_\e\| <\e.
$$
Therefore, by compactness,  for a sequence $0<\e_j \to 0$,
$$
q_{\e_j} \to q_0 \in  \gamma\cap  \widetilde P,
$$
which is false. Hence
 $\widetilde P_c^+ \cap\widetilde P_c^- =\emptyset$.
 \end{proof}

\section{Examples A and B}\label{AB}

\subsection*{A General Construction}
For $0<\a<\b<\infty$    let
$$ C(\a,\b)= \big\{(\l,x): 0< \a\l < x < \b \l \text{ or } 0 > \a\l > x > \b \l\}\cup\{(0,0)\},$$
a  double cone in the first and third quadrants.
 Then there exists $\om:\RR^2 \to \RR$ with the following properties:
 \begin{quotation}\noindent
(a) $\om(\l,x) =0$ if $|x| \geq \alpha|\l|/2$, in particular, $\om=0$ on $C(\a,\b)$;
\\(b) $\l\,\om(\l,x) \geq 0$ on $\RR^2$;
\\ (c)  $\om(\l,0) = \l,\quad \l \in \RR$;
\\(d)  $\om$ is infinitely differentiable  on $\RR^2\setminus \{(0,0)\}$;
\\(e)  $\om$ is globally Lipschitz continuous  $\RR^2$.
\end{quotation}
To see this, let
$\varpi:\RR \to \RR$ be an infinitely differentiable even function which is non-increasing on $[0,\infty)$ with $\varpi(0) = 1$ and $\varpi (r) = 0$ for all $r \geq \alpha/2$.
Then, for $x\in \RR$, let
$$
\om(\l,x) = \l\varpi\left(\frac{x}{\l}\right),~~ \l \neq 0,\quad \om(0,x) = 0.
$$
That $\om$ satisfies (a)-(d) follows immediately from the definition and the properties of $\varpi$.
Moreover, the  partial derivatives  at $(\l,x)$ are
\begin{subequations}\label{pds}
\begin{gather}
\label{pda}\pd_x \om(\l,x) = \varpi'\left(\frac{x}{\l}\right), \quad \pd_\l\om(\l,x) = \varpi\left(\frac{x}{\l}\right) - \left(\frac{x}{\l}\right) \varpi'\left(\frac{x}{\l}\right),
~~ \l \neq 0,\\
\label{pdb} \pd_x \om(0,x)= \pd_\l\om(0,x) = 0 \text { when }\l=0 \text{ and } x \neq 0,
\\
\label{pdc}\pd_x \om(0,0) = 0 \text{ and  }\pd_\l\om(0,0) = 1,\\
\intertext{ since $\om(\l,x)  = 0$ when $|\l| \leq |2|x|/\alpha$, and, for future reference, note}
\label{pdd} \pd_{x\l}\big(x\,\om(\l,x)\big) = \varpi\Big(\frac{x}{\l}\Big) - \Big(\frac{x}{\l}\Big)\varpi'\Big(\frac{x}{\l}\Big) - \Big(\frac{x}{\l}\Big)^2\varpi''\Big(\frac{x}{\l}\Big),~\l \neq 0,\\
\pd_{x\l}\big(x\,\om(\l,x)\big)= 0,\quad \l =0, ~~x\in \RR. \label{pde}
\end{gather}
\end{subequations}
Since $\varpi'(r) = 0,\,r\geq \alpha/2$, the partial derivatives of $\om$ are uniformly bounded in $\RR^2\setminus \{(0,0)\}$, and  property (e) follows.

\begin{remark}\label{important}
Note from \eqref{pdb} and \eqref{pdc}  that  $\pd_\l\om$ is not continuous at $(0,0)$ and  from \eqref{pda} and \eqref{pdc}  that  $\pd_x\om$ is not continuous at $(0,0)$. However,
$(\l,x) \mapsto x\,\om(\l,x)$ is continuously differentiable on $\RR^2$ but,   from \eqref{pdd}, \eqref{pde}, the mixed partial derivative  $\pd_{\l x} \big(x\,\om (\l,x)\big)$ is not continuous at $(0,0)$.
 \qed\end{remark}
\begin{definitionh*}Let $D^+$ and $D^-$ denote the two disjoint components of the complement of $ C(\a,\b)$ which contain the positive and negative $\l$-axes respectively. Then say that a set $G$ satisfies hypothesis $(\sH)$ if $G \subset C(\a,\b)$  is  closed, connected, and unbounded  in both half planes $\{ \l \geq 0\}$, and $\{\l \leq 0\}$, and $H^+ \cap H^- = \emptyset$ where  $H^\pm$ are the  connected components of $\RR^2 \setminus G$ with $D^\pm\subset H^\pm$, respectively.
\qed\end{definitionh*}
\begin{lemma}\label{lem17} If $G$ satisfies  $(\sH)$, then
$\om \geq 0$ on $H^+$ and $\om \leq 0$ on $H^-$
\end{lemma}
\begin{proof}This is immediate from properties (a) and (b) of $\omega$.
\end{proof}

\begin{lemma}\label{lemma16} When $G$ satisfies $(\sH)$  there is a globally Lipschitz continuous function $g: \RR^2 \to \RR$ which is infinitely differentiable on $\RR^2 \setminus \{(0,0)\}$, with   the property that $g(\l,0) = \l$~ for all $\l \in \RR$, and
$
g(\l,x) = 0 \text{ if and only if } (\l,x) \in G.$\end{lemma}
\begin{proof} Since $G$ is closed, by Whitney's lemma there exists a non-negative, infinitely differentiable function $h:\RR^2 \to [0,\infty)$ such that $h(\l,x) = 0$ if and only if $(\l,x) \in G$ and every derivative of $h$ is zero at every point of $G$. Let $\hat h:\RR^2 \to \RR$ be defined by
$$\hat h(\l,x) = \left\{\begin{array}{c}  ~-h(\l,x), \quad (\l,x) \in H^-\\ h(\l,x), \quad \text{otherwise}\end{array}\right\}, \text{ with $H^\pm$ defined in Definition $(\sH)$.}$$
In particular, $\hat h(\l,x)= \pm h(\l,x)$, $(\l,x) \in H^\pm$, $\hat h$ is infinitely differentiable on $\RR^2$, and $\hat h(\l,x) = 0$ if and only if $(\l,x) \in G$.
Now with $\om$ satisfying (a)-(e) above, let
\begin{equation}\label{nine}
g(\l,x) = x^2\,\hat h(\l,x) + \omega(\l,x).
\end{equation}
It follows from \eqref{pds}  that $g$ is infinitely differentiable on $\RR^2\setminus \{(0,0)\}$
and by Lemma \ref{lem17} $g$ satisfies the conclusions of the Lemma.
\end{proof}

\begin{prop}\label{prop19} For  $G$ satisfying  $(\sH)$, there is a continuously differentiable function   $r:\RR^2 \to \RR$ with $ r\in C^\infty\big(\RR^2\setminus \{0,0)\}\big)$,
such that
$|r(\l,x)|/|x| \to 0 \text{ as }0\neq |x| \to 0$  uniformly for $\l$ in bounded intervals,
and $G\setminus \{(0,0)\}$ is the set of non-trivial solutions of  $\l x = r(\l,x)$.
\end{prop}
\begin{proof}
For  $G$ satisfying $(\sH)$ and  the corresponding function $g$ in Lemma \ref{lemma16},  let
$$
r(\l,x) = x (\l -g(\l,x)), \quad(\l,x) \in \RR^2.$$  Then the smoothness of $\hat h$ and the properties of $\om$ in \eqref{pds} imply that $g$ is infinitely differentiable on $\RR^2\setminus\{(0,0)\}$ and, by Remark \ref{important}, $g$ is continuously differentiable on $\RR^2$ with $|r(\l,x)|/|x| \to 0 \text{ as }0\neq |x| \to 0$  uniformly for $\l$ in bounded intervals.  Moreover,  by construction,  non-trivial solutions of \eqref{EQN} are the zeros of $g$ with $x \neq 0$. So, by Lemma \ref{lemma16},   $G\setminus \{(0,0)\}$ is the set of  non-trivial solutions of  $\l x = r(\l,x)$ in $\RR^2$. This completes the proof.
\end{proof}
\begin{remark*}
Since, from Remark \ref{important}, the mixed partial derivative $\pd_{\l x}r$ is not continuous at $(0,0)$,  Theorem \ref{simple} does not apply to equation  $\l x = r(\l,x)$  in this situation. \qed.
\end{remark*}
\subsection*{Example  A}
Let $\widetilde P$  be the  unbounded connected set defined in  \eqref{Ftilde}
and let $G$  denote $\widetilde P$ rotated counter-clockwise about the origin through an angle $\pi/4$.  By Lemma \ref{newlem}, $G$ is connected, contains no non-trivial paths, and satisfies $(\sH)$ with   $\alpha = \sin(\pi/6)$ and
$\beta = \sin(5\pi/6)$.
With this choice of $G$, Example A is a special case of   Proposition \ref{prop19}. \qed

\subsection*{Example B}

This example shows that the global connected set $\sC$  given by Theorem \ref{godd} need not be path connected even when  all the operators involved are infinitely differentiable and,  by Theorem \ref{simple}, locally there is bifurcation from a simple eigenvalue.

Let  three disjoint connected sets  be defined by
 \begin{equation}\label{G}\begin{split}
 \sL &= \{0\}\times (-1/2,1/2),\\
  \sC^+ &= \big((0, 1/2) +  (\widetilde P\cap ([0, \infty) \times \RR) \big) \subset [0,\infty) \times [1/4,3/4],
\\
 \sC^- &= \big((0,-1/2) +  (\widetilde P \cap ((-\infty, 0] \times \RR)\big)\subset (-\infty,0]) \times [-3/4, -1/4].
 \end{split}\end{equation}
   Then $\sL$  is the smooth curve $\{0\}\times (-\uhalf,\uhalf)$, and $\sC^\pm$ are  closed, unbounded, connected sets  in the first and third quadrants respectively with $(0,\pm \uhalf) \in \sC^\pm$ and all  path-connected components of $\sC^+\cup \sC^-$ are singletons.
 Let $\ol\sC $ be their union
 $$
 \ol\sC=\sL \cup\sC^+ \cup \sC^-.
$$
If $E^-$ and $E^+$ denote the connected component of $\RR^2 \setminus \ol \sC$ which contains $(-\infty,0]\times \{0\}$ and $[0, \infty)\times \{0\}$, respectively, it follows from the argument for Lemma \ref{newlem} that
$E^+ \cap E^- = \emptyset$.
 By Whitney's result there exists a non-negative, infinitely differentiable function $h$ on $\RR^2$ which is zero only on the closed set
 $\ol \sC$, and  at each point of $\ol \sC$ all the derivatives of $h$ are zero. Let
 $$\tilde h(\l,x) = \left\{\begin{array}{c}  ~-h(\l,x), \quad (\l,x) \in E^-\\ h(\l,x), \quad \text{otherwise}\end{array}\right\},$$
 so that $\tilde h \geq 0$ on $E^+$.

 Now let $\tilde \om: \RR \to \RR$ be an infinitely differentiable even function with $\tilde \om(0)=1$, $\tilde \om$ decreasing on $[0, 1/4]$ and $\tilde \om(x) = 0$ when $|x|\geq 1/4$, let
 $ \tilde g(\l,x) = x^2 \tilde h(\l, x)+ \l \tilde \om (x)$.
Finally let
$
r(\l,x) = x(\l - \tilde g(\l,x)).
$
Then the set of non-trivial solutions of $\l x =  r(\l,x )$ coincide with the non-trivial solution set of  $\tilde g(\l,x)=0$ which is  the set $\ol \sC\setminus \{(0,0)\}$.
This completes the justification of Example B.
\qed

\section{Example C}

 Example C  is a simplified version of   B\"ohme's  example \cite{bohme} with
   added structure to ensure  that all path-connected sets of non-trivial solutions are singletons.

\begin{center}
\includegraphics[width=2.5in,height=2.2in]{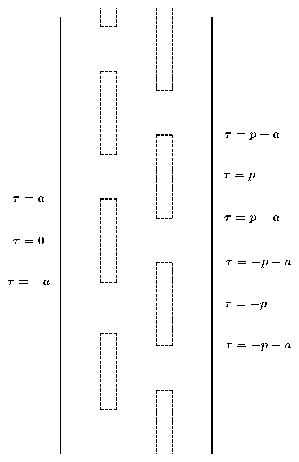}
\end{center}

According to Bing \cite[Ex.\,2,\,p.\,48]{bingconcerning} there exists a  hereditarily indecomposable continuum, $H$ say, which separates the plane. Let $\Omega$  be a bounded component of $\RR^2 \setminus H$ and $\pd \Omega$ its boundary. Then
 $\pd \Omega \subset H$,
 since points which are not in $H$ (which is closed)  are interior points of their connected component in $\RR^2 \setminus H$.

 Without loss of generality, suppose that in the $(\varsigma,\tau)$-plane
\begin{equation}\label{omega}\Om \subset [-\uf\pi,\uf\pi] \times [-a,a]~~ \text{ and }~~ \ol\Om \cap [-\uf\pi,\uf\pi] \times \{\pm a\} \neq \emptyset, \quad a>0.
\end{equation}
Denote by $S$ the strip $[-\pi,\pi]\times \RR$ and, with $a<p<2a$, consider  two parallel columns of copies of $\Omega$, arranged periodically with period $2p$ in the $\tau$ direction,
centred on the lines $\varsigma= \pm \pi/2$, and with height $2a$,  as illustrated  in the diagram. The copies of  $\Om$ in the right column  are translates through $(\pi, p)$ of those on the left.  (Apart from being open, connected and satisfying \eqref{omega}, nothing is known about the shape of $\Omega$, so the diagram is  for illustration only.) Let $\widehat \Omega$ denote the union of all the copies of $\Om$ in this arrangement.   The key to what follows is the property of $\widehat \Omega$ that, for all $\tau \in \RR$, the set $\{\varsigma: (\varsigma,\tau) \in \widehat \Omega\}$ has strictly positive measure.


Now by Whitney's result there exists $\psi: \RR^2 \to \RR$ which is infinitely differentiable,  $\psi>0$ on $\RR \setminus \pd \widehat \Omega$, and $\psi$ and all its derivatives are zero on $\pd \widehat \Omega$. There is no loss of generality in assuming that $\psi$ is $2p$-periodic in $\tau$ and equals 1 in the two strips  $[7\pi/8,\pi] \times \RR$ and  $[-\pi,-7\pi/8] \times \RR$.
Now, for $(\varsigma ,\tau) \in S$, let
\begin{equation}\begin{split}\psi^-(\varsigma,\tau) &= - \psi(\varsigma,\tau) \text{ when }(\varsigma,\tau) \in \widehat \Omega,\quad  \text { and } \psi^-(\varsigma,\tau)= 0 \text{ otherwise},\\
\psi^+(\varsigma,\tau) &=  \psi(\varsigma,\tau) \text{ when }(\varsigma,\tau) \in S \setminus \widehat \Omega,\quad  \text { and } \psi^+(\varsigma,\tau)= 0 \text{ otherwise}.
\end{split}\end{equation}
Next define infinitely differentiable functions  $\kappa^\pm$ which are $2p$-periodic in $\tau\in \RR$ by
$$
\kappa^\pm(\tau) = \int_{-\pi}^{\pi} \psi^\pm(\varsigma,\tau)\,d\varsigma,\quad \tau \in \RR,
$$
where $\kappa_-(\tau) <0<\kappa^+(\tau),~\tau \in \RR$, and let
\begin{gather}\label{Phi1} \vp(\varsigma,\tau) = \kappa^+(\tau)\psi^-(\varsigma,\tau) -   \kappa^-(\tau) \psi^+(\varsigma,\tau).\\
\intertext{Then $\vp(\varsigma,\tau) = -\kappa^-(\tau)>0$ when $|\varsigma-\pi|<\pi/8$, $\vp$ is infinitely differentiable,  $\pd \widehat \Omega$ is the zero set of $\vp$, and}
\int_{-\pi}^\pi \vp(\varsigma,\tau)\,d\varsigma = 0 \text{ for all } \tau \in \RR.
\label{Phi2}\end{gather}
If $\Phi:S \to \RR$  is defined by
\begin{gather}\label{dPhi}\Phi(\varsigma,\tau) =  \int_{-\pi}^\varsigma \vp(s,\tau)\,ds,\quad (\varsigma,\tau) \in S,
\\ \intertext{then for $\tau \in \RR$,}
\notag\Phi(-\pi,\tau) = \Phi(\pi,\tau) = 0,\quad  \frac{\partial\Phi}{\partial \varsigma}( \varsigma, \tau) = -\kappa^-(\tau),~~|\varsigma-\pi|<\pi/8,
\\ \text{ and }\qquad
\notag\frac{\partial^k\Phi}{\partial \varsigma^k}(-\pi,\tau) = \frac{\partial^k\Phi}{\partial \varsigma^k}( \pi, \tau) = 0 \forall k \geq 2.
\end{gather}
With this in mind, an infinitely differentiable function $r:\RR^2 \to \RR$ can be defined by
putting $ r(0,0) = 0$ and, for   $(x,y) =\rho(\cos \vt,\sin \vt)$ in polar coordinates, let
\begin{align}\label{rPhi}r(x,y) =\hat r(\rho,\vt) &:= \exp\left(\frac{-1}{\rho^2}\right)\Phi\left(\frac{1}{\rho},\vt\right),~~ \rho>0, ~~ \vt \in [-\pi,\pi]. \end{align}
Then since \eqref{EQN2} can be written
$$\nabla\left(\frac{1}{2} \l\|(x,y)\|^2 -r(x,y)\right) =0,$$ its non-trivial solutions satisfy
\begin{equation*}\label{equaviv}
\frac{\partial ~}{\partial \rho}\left(\half \l\,\rho^2
 -\hat r (\rho,\vt) \right)  =0, \quad\quad \frac{\partial ~}{\partial \vt}\left(\half \l\,\rho^2
 -\hat r (\rho,\vt) \right)  =0
 \quad \rho>0,~~ \vt \in [-\pi,\pi].
\end{equation*}
By \eqref{dPhi} and \eqref{rPhi},  the second equation implies that
\begin{equation}\label{Killer} \vp\left(\frac{1}{\rho},\vt\right) =0, \quad \rho >0, \text{ which means that } \left(\frac{1}{\rho},\vt\right) \in \widehat \Omega.\end{equation}
  Therefore, from \eqref{Killer} and the construction of $\widehat \Omega$, it follows that all non-trivial solutions of \eqref{EQN2} in this example lie the union of an infinite set
  of disjoint scaled copies of $\pd \Om$ (see \eqref{omega}) and, since $\pd\Omega \subset H$,  all the path-connected components of the non-trivial solution set are singletons.
The existence of non-trivial solutions of \eqref{EQN2} with $\big(\l,( x,y)\big)$ near to $\big((0,(0,0)\big))$ is guaranteed by Theorem \ref{gradbifthy}. \qed
\begin{appendix}

\section{Appendix}
\label{Ax1}
 \begin{prop} \label{prop1}In a metric space $(M,d)$ let    $\sG = \{G_\alpha: \alpha \in \sA\}$  be an open cover of a connected set $A$. Then for any $\e>0$ and   $x,y \in A$, there exists a finite set $\{G_{\a_1}, \cdots, G_{\a_n}\}\subset \sG$ with
\begin{equation}\label{chain}
x\in G_{\a_1},~~ y \in G_{\a_n} ~~  \text{and} ~~ G_{\a_i} \cap G_{\a_{j}} \neq \emptyset ~~\iff ~~|i-j|\leq 1.
\end{equation}
\end{prop}
\begin{proof}
Fix  $x \in A$ and let $B \subset A$ be the set of $y \in A$ such that \eqref{chain} holds for an ordered finite subset of $\sG$. It is immediate from  \eqref{chain} that  $z \in B$ if  $z \in G_{{\a_n}}\cap A$. So $B$ is open in $A$.
Now suppose $z$ is in the closure of $B$ in $A$.

Since $\sG$ covers $A$, there exists $G \in \sG$  such that $z \in G$, and there exists $y \in B$ with $y\in G$.   Since  $y\in B \cap G$, there exists $ \{G_{\tilde\a_1}, \cdots, G_{\tilde\a_m}\}\subset \sG$  be such that \eqref{chain} holds. Let $k$ be the smallest element of  $\{1,\cdots, m\}$ for which $ G_{\tilde \a_k}\cap G \neq \emptyset$. Then
$ \{ G_{\tilde\a_{n_j}},1 \leq j \leq k\}\cup \{G\}$
   satisfies \eqref{chain} with $z$ instead of $y$.
So $z \in B$, which proves that $B$ is closed, as well as open,  in $A$. Since $A$ is connected in $M$ and $B \neq \emptyset$, because $x\in B$, it follows that   $B=A$.
\end{proof}
\begin{corollary}\label{corex}
For given  $x,y \in A$, a connected set in   $(M,d)$,  and all $\e>0$ sufficiently small, there  is a discrete set of points   $\{x_1,\cdots,x_{n_\e}\}$ in  $ A$ with $x_1 = x,~x_{n_\e} = y$ and
\begin{equation}\label{balls}
 B_\e(x^\e_i) \cap B_\e(x^\e_j)\neq \emptyset \iff |i-j|\leq 1,   ~~  1 \leq i,j\leq n_\e,
\end{equation}  where $B_\e(a)$ is the open ball with radius $\e$ centred at  $a$ in $M$.
\end{corollary}
\begin{proof}
By Proposition  \ref{prop1} with  $\sG = \{B_{\e}(a): a\in A\}$ there exists $a^\e_j\in A$, $1\leq j \leq m_\e$ such that
$$
x \in B_{\e}(a_1^\e),~~y \in B_{\e}(a_{m_\e}^\e), ~~ B_\e(a^\e_i) \cap B_\e(a^\e_j)\neq \emptyset \iff |i-j|\leq 1.
$$
 Now let
$$m^\e_x = \sup\{j: B_{\e}(x) \cap B_\e(a^\e_j)\neq \emptyset\}\text{  and } m^\e_y = \inf\{j: B_{\e}(y) \cap B_\e(a^\e_j)\neq \emptyset\},$$
and, with  $\e>0$ sufficiently small that $m^\e_y > m^\e_x$, let $n_\e = m^\e_y-m^\e_x+3$ and put
$$x^\e_1 = x, \quad x^\e_2= a^\e_{m_x},~~ x^\e_3 = a^\e_{m_x +1}, \cdots , ~x^\e_{n-1} = a^\e_{m_y},\quad x^\e_{n_\e} = y,
$$
to achieve the required result.
\end{proof}
\begin{corollary}\label{piecewise}  In the special case when $(M,d)$ is a normed linear space and  the    balls   $B_\e(x^\e_i)$ are as in Corollary \ref{corex},
  let  $L^\e_i$  be the straight line segments, $\{t x^\e_i+(1-t)x^\e_{i+1}: t \in [0,1]\}$, which join the  centres of consecutive balls.
Then
$$
L^\e_i \cap L^\e_{i+1}=\{e^\e_{i+1}\}~~ \text{ and }~~ L^\e_i\cap L^\e_j = \emptyset ~\text{ when }i+1<j\leq n_\e,~~  i\geq 1.
 $$
Consequently, $L^\e: = \cup_{i=1}^{n-1}L^\e_i$ is
a continuous, piecewise-linear,  non-self-intersecting curve joining $x$ to $y$.
\end{corollary}
\begin{proof} First suppose that  $z\in L^\e_i \cap L^\e_{i+1}$ and   $z\neq e^\e_{i+1}$.
Then
$$
z=(1-s)x^\e_{i+1}+sx^\e_{i+2} = tx^\e_i +(1-t)x^\e_{i+1},\quad s,~t \in (0,1],
$$
whence
$t(x^\e_i - x^\e_{i+1}) = s(x^\e_{i+2} - x^\e_{i+1})$, and so $s \neq t$ because $x_i^\e \neq x_{i+2}^\e$. If $s<t$,
$$
2\e \leq \|x_i^\e-x^\e_{i+2}\|  = \big(1-(s/t)\big)\|x_{i+2}^\e-x^\e_{i+1}\| <2\e,
$$
a contradiction, and if $t<s$,
$$
2\e \leq \|x_i^\e-x^\e_{i+2}\|  = \big(1-(t/s)\big)\|x_{i}^\e-x^\e_{i+1}\| <2\e,
$$
which is also false. This proves that $L^\e_i \cap L^\e_{i+1}=\{e^\e_{i+1}\}$ for all $i$.

Suppose  $ z \in L^\e_i\cap L^\e_j$ for  $i\geq 1$ and $i+1<j\leq n_\e-1$. Then, by \eqref{balls},
\begin{align*}
\|x^\e_i&-x^\e_{i+1}\| <2\e,~~ \|x^\e_j-x^\e_{j+1}\| <2\e, ~~ \|x^\e_i-x^\e_j\|\geq 2\e, ~~\|x^\e_{i+1}-x^\e_{j+1}\|\geq 2\e, \\\intertext{and}
z&=sx^\e_i+(1-s)x^\e_{i+1} = tx^\e_j +(1-t)x^\e_{j+1},\quad s,~t \in [0,1],\\
&=(1-s')x^\e_i+ s'x^\e_{i+1} = (1-t')x^\e_j +t'x^\e_{j+1},~~ s'=1-s,~t'=1-t.
\end{align*}
Therefore
$ x^\e_{i+1} +s(x^\e_i-x^\e_{i+1}) = x^\e_{j+1} +t(x^\e_j-x^\e_{j+1})$, which implies
$$2\e \leq \|x^\e_{i+1}-x^\e_{j+1}\|\leq s\|x^\e_i-x^\e_{i+1}\| +t\|x^\e_j-x^\e_{j+1}\| < 2\e(s+t),
$$
and hence $s+t > 1$.
Also $x^\e_i-x^\e_j = s'(x^\e_i-x^\e_{i+1}) + t'(x^\e_{j+1}-x^\e_j)$ and hence
$$
2\e\leq \|x^\e_i-x^\e_j\| \leq s'\|x^\e_{i}-x^\e_{i+1}\| + t'\|x^\e_{j+1}-x^\e_{j}\|  <2\e(s'+t'),
$$
from which it follows that $s'+t'>1$, equivalently, $s+t<1$, which is a contradiction.
Since the distinct line segments $L^\e_i$ joining centres of balls do not intersect, their union $L^\e$  is a continuous, piecewise-linear,  non-self-intersecting curve joining $x_1$ to $x_n$. \end{proof}
\begin{definition}\label{linearchain}
A linear chain $\sG$  is an ordered, finite collection of open  sets  with $G_i \cap G_j \neq \emptyset$ if and only if $|i-j|\leq 1$.
The $G_i$, which may not be connected, are   the links of $\sG$ and an $\e$-linear chain is a linear chain with  links of diameter less that $\e$.

If, for all $\e>0$,  a set $A$ can be covered by an $\e$-linear chain,  $A$ is said to be chained.
A chained hereditarily indecomposable continuum is called a pseudo-arc. \qed\end{definition}
\begin{remark}\label{mostcontinua}
According to \cite[Thm. 10]{bingembedding}, in the plane most bounded continua  (in the sense of the Baire Category Theorem in the complete metric space of continua with the Hausdorff metric) are pseudo-arcs  for which
the links of the covering $\e$-linear chains are open disks of diameter less than $\e$. This  statement is  stronger than \eqref{balls}  because  there the chain and its length $n$ depend on $\e$,  $x$ and $y$. When $A$
is compact $n$ is bounded depending only on $\e$ (see \cite{beer}), but the links of the chain
 still depend on $x,\,y$.  \qed
\end{remark}

 \end{appendix}

{J. F. Toland\\{\noindent Department of Mathematical Sciences,\\ University of Bath,\\ Bath,\\
 BA2 7AY UK\\ {\tt{email:~~masjft@bath.ac.uk}}}

\ed